\documentclass{article}

\usepackage{amsmath,amsthm,amsopn,amstext,amscd,amsfonts,amssymb,verbatim}
\usepackage[numbers]{natbib}

\usepackage{graphicx}
\usepackage{verbatim,color}
\usepackage{caption}
\usepackage{subcaption}
\usepackage{tcolorbox}
\usepackage{mathtools}

\def\tr{\mathop{\rm tr}\nolimits}

\def\im {\mathop{\rm Im}\nolimits}
\def\re{\mathop{\rm Re}\nolimits}
\def\etr{\mathop{\rm etr}\nolimits}

\renewenvironment{abstract}
                 {\vspace{6pt}
                  \begin{center}
                  \begin{minipage}{5in}
                  \centerline{\textbf{Abstract}}
                  \noindent\ignorespaces
                 }
                 {\end{minipage}\end{center}}
                 
\newtheorem{theorem}{\textbf{Theorem}}[section]
\newtheorem{corollary}{\textbf{Corollary}}[section]

\newtheorem{proposition}{\textbf{Proposition}}[section]
\theoremstyle{definition}
\newtheorem{definition}{\textbf{Definition}}[section]
\newtheorem{example}{\textbf{Example}}[section]
\newtheorem{remark}{\textbf{Remark}}[section]

\title{\Large \textbf{Matrix variate $p$-value in MANOVA}}
\author{
  \textbf{Jos\'e A. D\'{\i}az-Garc\'{\i}a} \thanks{Corresponding author\newline
    }\\
  {\normalsize Universidad Aut\'onoma de Chihuahua} \\
  {\normalsize Facultad de Zootecnia y Ecolog\'{\i}a} \\
  {\normalsize Perif\'erico Francisco R. Almada Km 1, Zootecnia} \\
  {\normalsize 33820 Chihuahua, Chihuahua, M\'exico}\\
  {\normalsize E-mail: jadiaz@uach.mx}\\
  \textbf{Francisco J. Caro-Lopera}\\
  {\normalsize University of Medellin} \\
  {\normalsize Faculty of Basic Sciences} \\
  {\normalsize Carrera 87 No.30-65} \\
  {\normalsize Medell\'{\i}n, Colombia} \\
  {\normalsize E-mail: fjcaro@udemedellin.edu.co} \\[2ex]
}
\date{}
\begin{document}
\maketitle

\begin{abstract}
The distribution functions of the matricvariate beta type I and II distributions are studied under real normed division algebras. The unified approach for real, complex, quaternions and octonions, also considers general properties and highlights the potential application of the exact emerging upper probabilities $P(\mathbf{B} > \mathbf{\Omega})$ and $P(\mathbf{F} > \mathbf{\nabla})$. In this setting, the matrix probabilities arise naturally as univariate extensions into the so termed matrix variate $p$-values. Then, a new criterion for the general multivariate linear hypothesis test can be proposed under a simple heuristic interpretation. The new technique can be applied in a number of classical statistical tests. In particular, the multivariate analysis of variance (MANOVA) is illustrated in two well known scenarios, and the performance of our exact method is compared with the existing approximated criteria.
\end{abstract}                 

\section{Introduction}\label{sec:1}

The multivariate linear model takes the form
$$
  \mathbf{Y} = \mathbf{X} \mathbb{B} + \mathbf{E},
$$
where $\mathbf{Y} \in \Re^{n \times m}$ and $\mathbf{E} \in \Re^{n \times m}$ are random matrices, $\mathbf{X} \in \Re^{n \times p}$ is the design matrix or the regression matrix of rank $r \leq p$ and  $n \geq m + r$ ; and $\mathbb{B} \in \Re^{p \times m}$ involves the unknown  parameters termed regression coefficients. We shall assume that $\mathbf{E} \sim \mathcal{N}_{n \times m}(\mathbf{0}, \mathbf{I}_{n} \otimes \mathbf{\Sigma})$ then $\mathbf{Y} \sim
\mathcal{N}_{n \times m}(\mathbf{X}\mathbb{B}, \mathbf{I}_{n} \otimes \mathbf{\Sigma})$, see \citet[p. 430]{mh:05}). Here $\otimes$ denotes the Kronecker product; where $\mathbf{\Sigma} \in \Re^{m \times m}$, $\mathbf{\Sigma} > \mathbf{0}$ (positive definite matrix). Given $\mathbf{N} \in \Re^{q \times n}$ of known constants, then for estimable $\mathbf{M}\mathbb{B}$, the maximum likelihood or the least square estimate of $\mathbf{N}\mathbb{B}$ is given by
$$
  \widehat{\mathbf{N}\mathbb{B}} \equiv \mathbf{N}\widehat{\mathbb{B}} = \mathbf{N}
  (\mathbf{X}'\mathbf{X})^{-}\mathbf{X}'\mathbf{Y}= \mathbf{N}
  \mathbf{X}^{+}\mathbf{Y},
$$
where $\mathbf{A}^{-}$ is any generlised inverse of $\mathbf{A}$ (this is, $\mathbf{A} = \mathbf{AA}^{-}\mathbf{A}$) and $\mathbf{X}^{+}$ is the Moore-Penrose generalised inverse of $\mathbf{X}$.

We focus on testing the general multivariate linear hypothesis
\begin{equation}\label{hypo}
  H_{0}: \mathbf{C}\mathbb{B}\mathbf{M} =\mathbf{H} \quad vs \quad H_{a}: \mathbf{C}\mathbb{B}\mathbf{M} \neq
  \mathbf{H},
\end{equation}
where $\mathbf{C} \in \Re^{q \times p}$  of rank $q \leq p$, $\mathbf{M} \in \Re^{m \times g}$ of rank $g \leq m$ and $\mathbf{H} \in \Re^{q \times g}$, of rank = $\min(q,g)$ are matrices of known constants. The matrix $\mathbf{C}$ determines the hypothesis among the elements of the parameter matrix columns, while the matrix $\mathbf{M}$ allows hypothesis among the different response parameters. The matrix $\mathbf{M}$ plays a role in profile analysis, for example; in ordinary hypothesis test it is taken to be the identity matrix, $\mathbf{M} = \mathbf{I}_{m}$. 

The matrix of sum of squares and sum of products due to the hypothesis is given by
$$
  \mathbf{S}_{\mathbf{H}} = (\mathbf{C}\widehat{\mathbb{B}}\mathbf{M} - \mathbf{H})'(\mathbf{C}(\mathbf{X}'
  \mathbf{X})^{-}\mathbf{C}')^{-1}(\mathbf{C}\widehat{\mathbb{B}}\mathbf{M} - \mathbf{H}).
$$
where $\mathbf{A}'$ denotes the transpose of $\mathbf{A}$. The matrix of sums of squares and sums of products due to the error is 
$$
 \mathbf{S}_{\mathbf{E}} =  \mathbf{M}'\mathbf{Y}'\left(\mathbf{I}_{n} - \mathbf{XX}^{-}\right)\mathbf{YM}.
$$
where under the null hypothesis $H_{0}: \mathbf{C}\mathbb{B}\mathbf{M} = \mathbf{0}$, $\mathbf{S}_{\mathbf{H}} \in \Re^{g \times g}$ is Wishart distributed with $\nu_{H}$ degrees of freedom, $\mathbf{S}_{\mathbf{H}} \sim \mathcal{W}_{g}(\nu_{H}, \mathbf{M}'\mathbf{\Sigma M})$ and $\mathbf{S}_{\mathbf{E}} \sim \mathcal{W}_{g}(\nu_{E}, \mathbf{M}'\mathbf{\Sigma M})$. Specifically, $\nu_{H}$ and $\nu_{E}$ denote the degrees of freedom of the hypothesis and error, respectively.

Under the intersection union principle and the likelihood ratio test, \citet{r:57} and \citet{w:32} proposed diverse criteria for hypothesis testing (\ref{hypo}).

Now, let $\theta_{1}, \dots, \theta_{m}$ and $\lambda_{1}, \dots, \lambda_{m}$ be the eigenvalues of the matrices $\mathbf{S}_{\mathbf{H}}(\mathbf{S}_{\mathbf{H}}+\mathbf{S}_{\mathbf{E}})^{-1}$ and $\mathbf{S}_{\mathbf{H}}\mathbf{S}_{\mathbf{E}}^{-1}$,
respectively. Several authors have proposed a number of different criteria for testing the multivariate general linear hypothesis, see \citet{k:83},
\citet{re:02} and \citet{dgcl:08}. All these test statistics can be represented as \textbf{\textit{functions}} of the $s = \min(m, \nu_{H})$ non-zero eigenvalues $\lambda's$ and/or $\theta's$, taking in mind that $\lambda_{i} = \theta_{i}/(1 -\theta_{i})$ and $\theta_{i} = \lambda_{i}/(1 + \lambda_{i})$, $i = 1, \dots, s$. As pointed out by Pillai, 
\medskip
\begin{center} 
\begin{minipage}[c]{5in}
  \small{\textit{The choice of any specific function of the eigenvalues, as a basic for test criteria, has so far been made on additional considerations which are heuristic, see \citet{p:55}.}}
\end{minipage}
\end{center}
\medskip
In this heuristic setting, a motivation of extending the $p$-value into the matrix variate test, provides a natural arising of a new criterion for the general multivariate linear hypothesis test. Moreover, searching for a unified field theory, we shall find the distribution functions of $\mathbf{B}$ (matricvariate beta type I)  and $\mathbf{F}$ (matricvariate beta type II) for real normed division algebras, under null hypothesis $H_{o}$. Furthermore, the corresponding upper probabilities ($P(\mathbf{B} > \mathbf{\Omega})$ and $P(\mathbf{F} > \mathbf{\nabla})$, now enriched under a $p$-value meaning, are easily obtained by establishing some basic properties of the distributions of the matrices $\mathbf{B}$ and $\mathbf{F}$. This new approach can be applied to a number of classical tests, namely, for testing the general linear hypothesis (\ref{hypo}) in two MANOVA problems from the statistical literature.
 
\section{Preliminaries results}

Some basic results about real normed division algebras, jacobians, and multivariate gamma and beta functions are outlined. In adition, the matricvariate beta type I and II distributions on real normed division algebras are defined and two basic properties are studied.

\subsection{Real normed division algebras and multivariate functions}

A detailed discussion of real normed division algebras can be found in
\cite{b:02} and \cite{E:90}. For convenience, we shall introduce some conventions, although in
general we adhere to standard notation forms.

For our purposes: Let $\mathbb{F}$ be a field. An \emph{algebra}
$\mathfrak{F}$ over $\mathbb{F}$ is a pair $(\mathfrak{F};m)$, where $\mathfrak{F}$ is a
\emph{finite-dimensional vector space} over $\mathbb{F}$ and \emph{multiplication} $m :
\mathfrak{F} \times \mathfrak{F} \rightarrow A$ is an $\mathbb{F}$-bilinear map; that is, for
all $\lambda \in \mathbb{F},$ $x, y, z \in \mathfrak{F}$;
\begin{eqnarray*}
  m(x, \lambda y + z) &=& \lambda m(x; y) + m(x; z) \\
  m(\lambda x + y; z) &=& \lambda m(x; z) + m(y; z).
\end{eqnarray*}
Two algebras $(\mathfrak{F};m)$ and $(\mathfrak{E}; n)$ over $\mathbb{F}$ are said to be
\emph{isomorphic} if there is an invertible map $\phi: \mathfrak{F} \rightarrow \mathfrak{E}$
such that for all $x, y \in \mathfrak{F}$,
$$
  \phi(m(x, y)) = n(\phi(x), \phi(y)).
$$
By simplicity, we write $m(x; y) = xy$ for all $x, y \in \mathfrak{F}$.

Let $\mathfrak{F}$ be an algebra over $\mathbb{F}$. Then $\mathfrak{F}$ is
said to be
\begin{enumerate}
  \item \emph{alternative} if $x(xy) = (xx)y$ and $x(yy) = (xy)y$ for all $x, y \in \mathfrak{F}$,
  \item \emph{associative} if $x(yz) = (xy)z$ for all $x, y, z \in \mathfrak{F}$,
  \item \emph{commutative} if $xy = yx$ for all $x, y \in \mathfrak{F}$, and
  \item \emph{unital} if there is a $1 \in \mathfrak{F}$ such that $x1 = x = 1x$ for all $x \in \mathfrak{F}$.
\end{enumerate}
If $\mathfrak{F}$ is unital, then the identity 1 is uniquely determined.

An algebra $\mathfrak{F}$ over $\mathbb{F}$ is said to be a \emph{division
algebra} if $\mathfrak{F}$ is nonzero and $xy = 0_{\mathfrak{F}} \Rightarrow x =
0_{\mathfrak{F}}$ or $y = 0_{\mathfrak{F}}$ for all $x, y \in \mathfrak{F}$.

The term ``division algebra", comes from the following proposition, which
shows that, in such an algebra, left and right division can be unambiguously performed.

Let $\mathfrak{F}$ be an algebra over $\mathbb{F}$. Then $\mathfrak{F}$ is a
division algebra if, and only if, $\mathfrak{F}$ is nonzero and for all $a, b \in
\mathfrak{F}$, with $b \neq 0_{\mathfrak{F}}$, the equations $bx = a$ and $yb = a$ have unique
solutions $x, y \in \mathfrak{F}$.

In the sequel we assume $\mathbb{F} = \Re$ and consider classes of division
algebras over $\Re$ or ``\emph{real division algebras}" for short.

We introduce the algebras of \emph{real numbers} $\Re$, \emph{complex numbers}
$\mathfrak{C}$, \emph{quaternions} $\mathfrak{H}$ and \emph{octonions} $\mathfrak{O}$. Then, if
$\mathfrak{F}$ is an alternative real division algebra, then $\mathfrak{F}$ is isomorphic to
$\Re$, $\mathfrak{C}$, $\mathfrak{H}$ or $\mathfrak{O}$.

Let $\mathfrak{F}$ be a real division algebra with identity $1$. Then
$\mathfrak{F}$ is said to be \emph{normed} if there is an inner product $(\cdot, \cdot)$ on
$\mathfrak{F}$ such that
$$
  (xy, xy) = (x, x)(y, y) \qquad \mbox{for all } x, y \in \mathfrak{F}.
$$
If $\mathfrak{F}$ is a \emph{real normed division algebra}, then $\mathfrak{F}$ is isomorphic to
$\Re$, $\mathfrak{C}$, $\mathfrak{H}$ or $\mathfrak{O}$.

There are exactly four normed division algebras: real numbers ($\Re$), complex
numbers ($\mathfrak{C}$), quaternions ($\mathfrak{H}$) and octonions ($\mathfrak{O}$), see
\cite{b:02}. We take into account that, $\Re$, $\mathfrak{C}$,
$\mathfrak{H}$ and $\mathfrak{O}$ are the only normed division algebras; furthermore, they are
the only alternative division algebras.

Let $\mathfrak{F}$ be a division algebra over the real numbers. Then
$\mathfrak{F}$ has dimension either 1, 2, 4 or 8. In other branches of mathematics, the
parameters $\alpha = 2/\beta$ and $t = \beta/4$ are used, see \cite{er:05} and \cite{k:84},
respectively.

Finally, observe that

\medskip

\begin{tabular}{c}
  $\Re$ is a real commutative associative normed division algebras, \\
  $\mathfrak{C}$ is a commutative associative normed division algebras,\\
  $\mathfrak{H}$ is an associative normed division algebras, \\
  $\mathfrak{O}$ is an alternative normed division algebras. \\
\end{tabular}

Let ${\mathcal L}^{\beta}_{m,n}$ be the set of all $n \times m$ matrices of rank $m \leq n$
over $\mathfrak{F}$ with $m$ distinct positive singular values, where $\mathfrak{F}$ denotes a
\emph{real finite-dimensional normed division algebra}. In particular, let $GL(m,\mathfrak{F})$
be the space of all invertible $m \times m$ matrices over $\mathfrak{F}$. Let $\mathfrak{F}^{n
\times m}$ be the set of all $n \times m$ matrices over $\mathfrak{F}$. The dimension of
$\mathfrak{F}^{n \times m}$ over $\Re$ is $\beta mn$.

Let $\mathbf{A} \in \mathfrak{F}^{n \times m}$, then $\mathbf{A}^{*} =
\overline{\mathbf{A}}^{T}$ denotes the usual conjugate transpose. Denote by ${\mathfrak S}_{m}^{\beta}$ the real
vector space of all $\mathbf{S} \in \mathfrak{F}^{m \times m}$ such that $\mathbf{S} =
\mathbf{S}^{*}$. Let $\mathfrak{P}_{m}^{\beta}$ be the \emph{cone of positive definite
matrices} $\mathbf{S} \in \mathfrak{F}^{m \times m}$. Thus, $\mathfrak{P}_{m}^{\beta}$ consist
of all matrices $\mathbf{S} = \mathbf{X}^{*}\mathbf{X}$, with $\mathbf{X} \in
\mathfrak{L}^{\beta}_{m,n}$; then $\mathfrak{P}_{m}^{\beta}$ is an open subset of ${\mathfrak
S}_{m}^{\beta}$. Over $\Re$, ${\mathfrak S}_{m}^{\beta}$ consist of \emph{symmetric} matrices;
over $\mathfrak{C}$, \emph{Hermitian} matrices; over $\mathfrak{H}$, \emph{quaternionic
Hermitian} matrices (also termed \emph{self-dual matrices}) and over $\mathfrak{O}$,
\emph{octonionic Hermitian} matrices. Generically, the elements of $\mathfrak{S}_{m}^{\beta}$
are termed as \textbf{Hermitian matrices}, irrespective of the nature of $\mathfrak{F}$. The
dimension of $\mathfrak{S}_{m}^{\beta}$ over $\Re$ is $[m(m-1)\beta+2m]/2$. For any matrix $\mathbf{X} \in \mathfrak{F}^{n
\times m}$, $d\mathbf{X}$ denotes the\emph{ matrix of differentials} $(dx_{ij})$. Finally, we
define the \emph{measure} or volume element $(d\mathbf{X})$ when $\mathbf{X} \in
\mathfrak{F}^{m \times n}, \mathfrak{S}_{m}^{\beta}$, or $\mathfrak{D}_{m}^{\beta}$, see \cite{d:02}.

If $\mathbf{X} \in \mathfrak{F}^{n \times m}$ then $(d\mathbf{X})$ (the Lebesgue measure in
$\mathfrak{F}^{n \times m}$) denotes the exterior product of the $\beta mn$ functionally
independent variables
$$
  (d\mathbf{X}) = \bigwedge_{i = 1}^{n}\bigwedge_{j = 1}^{m}\bigwedge_{k =
  1}^{\beta}dx_{ij}^{(k)}.
$$
If $\mathbf{S} \in \mathfrak{S}_{m}^{\beta}$ then $(d\mathbf{S})$ (the Lebesgue measure in $\mathfrak{S}_{m}^{\beta}$) denotes the exterior product of the  $m(m-1)\beta/2 + m$
functionally independent variables,
$$
  (d\mathbf{S}) = \bigwedge_{i=1}^{m} ds_{ii}\bigwedge_{i < j}^{m}\bigwedge_{k = 1}^{\beta}ds_{ij}^{(k)}.
$$
Observe, that for the Lebesgue measure $(d\mathbf{S})$ defined thus, it is required that
$\mathbf{S} \in \mathfrak{P}_{m}^{\beta}$, that is, $\mathbf{S}$ must be a non singular
Hermitian matrix (Hermitian positive definite matrix). 

If $\mathbf{\Lambda} \in \mathfrak{D}_{m}^{\beta}$ then $(d\mathbf{\Lambda})$ (the
Lebesgue measure in $\mathfrak{D}_{m}^{\beta}$) denotes the exterior product
of the $\beta m$ functionally independent variables
$$
  (d\mathbf{\Lambda}) = \bigwedge_{i = 1}^{m}\bigwedge_{k = 1}^{\beta}d\lambda_{i}^{(k)}.
$$

Some Jacobians in the quaternionic case are obtained in \cite{lx:09}. We now cite some
Jacobians in terms of the parameter $\beta$, based on the works of \cite{k:84} and
\cite{d:02}.

\begin{proposition}\label{prop1} 
Let $\mathbf{X}$ and $\mathbf{Y} \in \mathfrak{S}_{m}^{\beta}$ be matrices of functionally independent variables. 
\begin{description}
  \item[i)] Let $\mathbf{Y} = \mathbf{AXA^{*}} + \mathbf{C}$, where $\mathbf{A} \in {\mathcal L}_{m,m}^{\beta} $ and
   $\mathbf{C} \in \mathfrak{S}_{m}^{\beta}$ are matrices of constants. Then
   \begin{equation}\label{hlt}
      (d\mathbf{Y}) = |\mathbf{A}^{*}\mathbf{A}|^{(m-1)\beta/2+1} (d\mathbf{X}).
   \end{equation}
  \item[ii)] Define $\mathbf{Y} = \mathbf{X}^{-1}$, Then
   \begin{equation}\label{hlt}
      (d\mathbf{Y}) = |\mathbf{X}|^{-(m-1)\beta-2} (d\mathbf{X}).
   \end{equation}
\end{description}
\end{proposition}

In addition, $\Gamma^{\beta}_{m}[a]$ denotes the \emph{multivariate Gamma function} for the
space $\mathfrak{S}_{m}^{\beta}$, which is defined by
\begin{eqnarray}
  \Gamma_{m}^{\beta}[a] &=& \displaystyle\int_{\mathbf{A} \in \mathfrak{P}_{m}^{\beta}}
  \etr\{-\mathbf{A}\} |\mathbf{A}|^{a-(m-1)\beta/2 - 1}(d\mathbf{A}) \nonumber\\
&=& \pi^{m(m-1)\beta/4}\displaystyle\prod_{i=1}^{m} \Gamma[a-(i-1)\beta/2]
\end{eqnarray}
where $\etr\{\cdot\} = \exp\{\tr(\cdot)\}$, $|\cdot|$ denotes the determinant and $\re(a)
> (m-1)\beta/2$, see \cite{gr:87}. 
The generalised Pochhammer symbol of
weight $\kappa$, defined as
$$
  [a]_{\kappa}^{\beta} = \prod_{i = 1}^{m}(a-(i-1)\beta/2)_{k_{i}} = \frac{\pi^{m(m-1)\beta/4}
    \displaystyle\prod_{i=1}^{m} \Gamma[a + k_{i} -(i-1)\beta/2]}{\Gamma_{m}^{\beta}[a]},
$$
where $\re(a) > (m-1)\beta/2 - k_{m}$ and
$$
  (a)_{i} = a (a+1)\cdots(a+i-1),
$$
is the standard Pochhammer symbol.

From \cite[p. 480]{h:55} the \emph{multivariate beta function} for the space
$\mathfrak{S}^{\beta}_{m}$, can be defined as

\begin{eqnarray}
    \mathcal{B}_{m}^{\beta}[a,b] &=& \int_{\mathbf{0}<\mathbf{S}<\mathbf{I}_{m}}
    |\mathbf{S}|^{a-(m-1)\beta/2-1} |\mathbf{I}_{m} - \mathbf{S}|^{b-(m-1)\beta/2-1}
    (d\mathbf{S}) \\ \label{beta1}
&=& \int_{\mathbf{R} \in \mathfrak{P}_{m}^{\beta}} |\mathbf{R}|^{a-(m-1)\beta/2-1}
    |\mathbf{I}_{m} + \mathbf{R}|^{-(a+b)} (d\mathbf{R}) \\ \label{beta2}
&=& \frac{\Gamma_{m}^{\beta}[a] \Gamma_{m}^{\beta}[b]}{\Gamma_{m}^{\beta}[a+b]},\label{beta3} 
\end{eqnarray}
where (\ref{beta1}) is obtained making the change of variable $\mathbf{R} = (\mathbf{I}-\mathbf{S})^{-1} -\mathbf{I}$ and, by Proposition \ref{prop1} ii), $(d\mathbf{S}) = |\mathbf{I}_{m}+\mathbf{R}|^{-(m-1)\beta-2}(d\mathbf{R})$, with Re$(a) > (m-1)\beta/2$ and Re$(b)> (m-1)\beta/2$. In addition, as consequence of (\ref{beta3}), we have that  $ \mathcal{B}_{m}^{\beta}[a,b] =  \mathcal{B}_{m}^{\beta}[b,a]$.

Finally, consider the definition and basic properties of the hypergeometric function with one matrix argument for real normed division algebras.

Fix complex numbers $a_{1}, \dots, a_{p}$ and $b_{1}, \dots, b_{q}$, and for all $1 \leq i \leq
q$ and $1 \leq j \leq m$ do not allow $-b_{i} + (j-1)\beta/2$ to be a nonnegative integer. Then
the \emph{hypergeometric function with one matrix argument} ${}_{p}F_{q}^{\beta}$ is defined to
be the real-analytic function on $\mathfrak{S}_{m}^{\beta}$ given by the series
\begin{equation}\label{fhoa}
    {}_{p}F_{q}^{\beta}(a_{1}, \dots, a_{p};b_{1}, \dots, b_{q}; \mathbf{X}) = \sum_{k=0}^{\infty}\sum_{\kappa}
  \frac{[a_{1}]_{\kappa}^{\beta} \cdots [a_{p}]_{\kappa}^{\beta}}{[b_{1}]_{\kappa}^{\beta} \cdots
  [b_{q}]_{\kappa}^{\beta}} \ \frac{C_{\kappa}^{\beta}(\mathbf{X})}{k!},
\end{equation}
where $C_{\kappa}^{\beta}(\mathbf{X})$ denotes the Jack polynomials, \citet{S97} also termed zonal polynomials, see \citet[Section 5.]{gr:87}.

In addition, for the convergence of hypergeometric series we have:
\begin{enumerate}
  \item If $p \leq q$ then the hypergeometric series (\ref{fhoa}) converges absolutely
    for all $\mathbf{X} \in \mathfrak{S}_{m}^{\beta}$.
  \item If $p = q+1$ then the series (\ref{fhoa}) converges absolutely for $||\mathbf{X}||=
  \max\{|\lambda_{i}|: i = 1, \dots, m \} < 1$, and diverges for $||\mathbf{X}|| >
    1$, where $\lambda_{1}, \dots \lambda_{m}$ are the $i$-th eigenvalues of $\mathbf{X}
    \in \mathfrak{S}_{m}^{\beta}$.
  \item If $p > q$ then the series (\ref{fhoa}) diverges unless it terminates.
\end{enumerate}

The sum of hypergeometric series is studied in term of integral properties for all $\mathbf{X} \in \mathfrak{S}_{m}^{\beta}$; indeed, for all $\mathbf{X} \in
\mathfrak{S}_{m}^{\beta, \mathfrak{C}}$. Where $\mathfrak{S}_{m}^{\beta, \mathfrak{C}}$ denotes the complexification $\mathfrak{S}_{m}^{\beta, \mathfrak{C}} = \mathfrak{S}_{m}^{\beta} + i \mathfrak{S}_{m}^{\beta}$ of $\mathfrak{S}_{m}^{\beta}$. That is, $\mathfrak{S}_{m}^{\beta, \mathfrak{C}}$ consist of all matrices $\mathbf{Z} \in (\mathfrak{F^{\mathfrak{C}}})^{m \times m}$ of the form $\mathbf{Z} = \mathbf{X} + i\mathbf{Y}$, with $\mathbf{X}, \mathbf{Y} \in \mathfrak{S}_{m}^{\beta}$. We refer to $\mathbf{X} = \re(\mathbf{Z})$ and $\mathbf{Y} = \im(\mathbf{Z})$ as the \emph{real and imaginary parts} of $\mathbf{Z}$, respectively. The \emph{generalised right half-plane} $\mathbf{\Phi} = \mathfrak{P}_{m}^{\beta} + i \mathfrak{S}_{m}^{\beta}$ in $\mathfrak{S}_{m}^{\beta,\mathfrak{C}}$ consists of all $\mathbf{Z} \in \mathfrak{S}_{m}^{\beta,\mathfrak{C}}$ such that $\re(\mathbf{Z}) \in \mathfrak{P}_{m}^{\beta}$, see \cite[p. 801]{gr:87}.

A detailed study on the hypergeometric function with one matrix argument for real normed division algebras is presented in \cite[Section 6, pp. 803-810]{gr:87} and in \citet{c:63} and \citet[Section 7.3]{mh:05} in real case and \citet[section 4 and Section 8]{j:64} in real and complex cases, respectively.

\subsection{Beta type I and II distributions}

\begin{definition}
\begin{description}
  \item[i)] The random matrix $\mathbf{U} \in \mathfrak{P}_{m}^{\beta}$ is said to have a matricvariate beta type I distribution, with parameters  Re$(a) > (m-1)\beta/2$ and Re$(b)> (m-1)\beta/2$, if its density function with respect to Lebesgue measure $(d\mathbf{U})$ in $\mathfrak{P}_{m}^{\beta}$ is
      
      \begin{equation}\label{dfbI}
        dF_{\mathbf{U}}(\mathbf{U}) = \frac{1}{\mathcal{B}_{m}^{\beta}[a,b]}  |\mathbf{U}|^{a-(m-1)\beta/2-1} |\mathbf{I}_{m} - \mathbf{U}|^{b-(m-1)\beta/2-1} (d\mathbf{U}),
      \end{equation}
      $\mathbf{0}<\mathbf{U}<\mathbf{I}_{m}$. This fact shall be denoted as $\mathbf{U} \sim \mathfrak{B}_{m}^{\beta}(a,b)$.
  \item[ii)]  The random matrix $\mathbf{F} \in \mathfrak{P}_{m}^{\beta}$ is said to have a matricvariate beta type II distribution with parameters  Re$(a) > (m-1)\beta/2$ and Re$(b)> (m-1)\beta/2$, if its density function is
      
      \begin{equation}\label{dfbII}
        dF_{\mathbf{F}}(\mathbf{F}) = \frac{1}{\mathcal{B}_{m}^{\beta}[a,b]} |\mathbf{F}|^{a-(m-1)\beta/2-1}
        |\mathbf{I}_{m} + \mathbf{F}|^{-(a+b)} (d\mathbf{F}),
      \end{equation}
      which exist with respect to Lebesgue measure $(d\mathbf{F})$ in $\mathfrak{P}_{m}^{\beta}$. We shall write that $\mathbf{F} \sim \mathcal{F}_{m}^{\beta}(a,b)$. 
\end{description}
\end{definition}

Observe that explicit forms for the Lebesgue measure when $\mathbf{S} \in \mathfrak{P}_{m}^{\beta}$ can be obtained in terms of Cholesky and spectral decomposition, among other, see \citet[Eqs. (2.11) and (2.12)]{dg:14} in the general case and see \citet{dggf:05} for real case.

\begin{theorem}\label{teo1}
Assume that $\mathbf{S} = \mathbf{I} - \mathbf{U}$, where  $\mathbf{U} \sim \mathfrak{B}_{m}^{\beta}(a,b)$. Then $\mathbf{S} \sim \mathfrak{B}_{m}^{\beta}(b,a)$.
\end{theorem}
\begin{proof}
The proof follows by observing that $\mathbf{U} = \mathbf{I} -\mathbf{S}$ and $(d\mathbf{U}) = (d\mathbf{S})$. 
\end{proof}

\begin{theorem}\label{teo2}
Suppose that $\mathbf{F} \sim \mathcal{F}_{m}^{\beta}(a,b)$ and define $\mathbf{R} = \mathbf{F}^{-1}$. Then $\mathbf{R} \sim \mathcal{F}_{m}^{\beta}(b,a)$.
\end{theorem}
\begin{proof}
This follows by noting that $\mathbf{F} = \mathbf{R}^{-1}$. By Proposition \ref{prop1} ii) we have that $(d\mathbf{F}) = |\mathbf{R}|^{-(m-1)\beta-2}(d\mathbf{R})$. 
\end{proof}

\section{Main results}

Our main goal is to find the upper probabilities $P(\mathbf{S} > \mathbf{\Omega})$ when $\mathbf{S}$ has a matricvariate beta type I and II distributions and $\mathbf{\Omega} \in \mathfrak{P}_{m}^{\beta}$.
First we shall study their corresponding distribution functions $F_{\mathbf{S}}(\mathbf{S}) = P(\mathbf{S} < \mathbf{\Omega})$ and then we obtain $P(\mathbf{S} > \mathbf{\Omega})$. Unfortunately as is established in \citet{c:63},
\medskip
\begin{center} 
\begin{minipage}[c]{5in}
  \small{\textit{for $m \geq 2$, $P(\mathbf{S} < \mathbf{\Omega}) \neq 1- P(\mathbf{S} > \mathbf{\Omega})$, since the set of $\mathbf{S}$ where neither of the relations $\mathbf{S} < \mathbf{\Omega}$ nor $\mathbf{S} > \mathbf{\Omega}$ holds is not of measure zero. \textbf{The complementary probabilities $P(\mathbf{S} > \mathbf{\Omega})$)  seem difficult to evaluate}. Also see \citet[p. 421]{mh:05}.}}
\end{minipage}
\end{center}
\medskip
For the beta type I distribution, the real incomplete beta function was obtained by \citet{c:63} (for $P(\mathbf{S} < \mathbf{\Omega})$ see \citet{ar:21}). Similarly, in the real case, the $P(\mathbf{S} > \mathbf{\Omega})$ was derived by \citet{ar:21} using a complex procedure, which is revisited in this work by elucidating a very simple  alternative approach. In terms of theorems \ref{teo1} and \ref{teo2} the corresponding expressions of $P(\mathbf{S} > \mathbf{\Omega})$ are straightforwardly obtained from their corresponding distribution functions in the real normed division algebra case.

\subsection{Matricvariate beta type I distribution}

\begin{theorem}\label{teo3}
  If $\mathbf{U}$ has the distribution (\ref{dfbI}), then the bounded lower probability of $\mathbf{U}$ is given by
  $$
    P(\mathbf{U} < \mathbf{\Omega}) = \frac{\mathcal{B}_{m}^{\beta}[a, (m-1)\beta/2 +1]}{\mathcal{B}_{m}^{\beta}[a,b]} |\mathbf{\Omega}|^{a} \hspace{5cm}
  $$
  $$
    \hspace{3cm}
    \times {}_{2}F_{1}^{\beta}(a,-b+(m-1)\beta/2+1;a+(m-1)\beta/2+1;\mathbf{\Omega}),
  $$
  where $\mathbf{0} < \mathbf{\Omega} < \mathbf{I}_{m}$. 
\end{theorem}
\begin{proof}
 This follows from Eq. (4.26) in \citet{dg:14}. In addition, as a consequence of the Euler relation (\citet{dg:14}), we have that
 
 \begin{equation}\label{k1}
   {}_{2}F_{1}^{\beta}(a,b;c;\mathbf{X}) = |\mathbf{I}_{m} - \mathbf{X}|^{c-a-b}{}_{2}F_{1}^{\beta}(c-a,c-b;c;\mathbf{X});
 \end{equation}
 alternatively, we obtain that
 $$
    P(\mathbf{U} < \mathbf{\Omega}) = \frac{\mathcal{B}_{m}^{\beta}[a, (m-1)\beta/2 +1]}{\mathcal{B}_{m}^{\beta}[a,b]} |\mathbf{\Omega}|^{a} |\mathbf{I}_{m} - \mathbf{\Omega}|^{b}\hspace{5cm}
 $$
 $$
    \hspace{3cm}
    \times {}_{2}F_{1}^{\beta}((m-1)\beta/2+1,a+b,a+(m-1)\beta/2+1;\mathbf{\Omega}).
 $$
 A third expression for the probability can be obtained from the following Euler relation
 
 \begin{equation}\label{k2}
   {}_{2}F_{1}^{\beta}(a,b;c;\mathbf{X}) = |\mathbf{I} - \mathbf{X}|^{-b}  {}_{2}F_{1}^{\beta}(c-a,b;c;-\mathbf{X}(\mathbf{I} - \mathbf{X})^{-1}).
 \end{equation}
 Hence
  $$
    P(\mathbf{U} < \mathbf{\Omega}) = \frac{\mathcal{B}_{m}^{\beta}[a, (m-1)\beta/2 +1]}{\mathcal{B}_{m}^{\beta}[a,b]} |\mathbf{\Omega}|^{a} |\mathbf{I}_{m} - \mathbf{\Omega}|^{b-(m-1)\beta/2-1}\hspace{5cm}
  $$
  $$
     \hspace{1cm}
     \times {}_{2}F_{1}^{\beta}((m-1)\beta/2+1,-b+(m-1)\beta/2+1;a+(m-1)\beta/2+1;-\mathbf{\Omega}(\mathbf{I}_{m} - \mathbf{\Omega})^{-1}).
  $$
\end{proof}
Now, in the prelude of the claimed solution of \citet{c:63} and the intricate derivation of \citet{ar:21}, the following unified field statements are also straightforward corollaries of the simplest lower probability.
\begin{corollary}\label{cor1}
  Assume that $\mathbf{U} \sim \mathfrak{B}_{m}^{\beta}(a,b)$ then
  $$
   P(\mathbf{U} > \mathbf{\Omega}) = \frac{\mathcal{B}_{m}^{\beta}[b, (m-1)\beta/2 +1]}{\mathcal{B}_{m}^{\beta}[a,b]} |\mathbf{I}_{m}-\mathbf{\Omega}|^{b} \hspace{5cm}
  $$
  \begin{equation}\label{pvBI1}
    \hspace{3cm}
    \times {}_{2}F_{1}^{\beta}(b,-a+(m-1)\beta/2+1;b+(m-1)\beta/2+1;\mathbf{I}_{m}-\mathbf{\Omega}),
  \end{equation}
  where $\mathbf{0} < \mathbf{\Omega} < \mathbf{I}_{m}$. Or alternatively
  $$
    P(\mathbf{U} > \mathbf{\Omega}) = \frac{\mathcal{B}_{m}^{\beta}[b, (m-1)\beta/2 +1]}{\mathcal{B}_{m}^{\beta}[a,b]} |\mathbf{I}_{m} - \mathbf{\Omega}|^{b} |\mathbf{\Omega}|^{a}\hspace{5cm}
  $$
  \begin{equation}\label{pvBI2}
        \hspace{3cm}
     \times {}_{2}F_{1}^{\beta}((m-1)\beta/2+1,a+b,b+(m-1)\beta/2+1;\mathbf{I}_{m}-\mathbf{\Omega}),
  \end{equation}
 where $\mathbf{0} < \mathbf{\Omega} < \mathbf{I}_{m}$. Or as
  $$
    P(\mathbf{U} > \mathbf{\Omega}) = \frac{\mathcal{B}_{m}^{\beta}[b, (m-1)\beta/2 +1]}{\mathcal{B}_{m}^{\beta}[a,b]} |\mathbf{I}_{m}-\mathbf{\Omega}|^{b} |\mathbf{\Omega}|^{a-(m-1)\beta/2-1}\hspace{5cm}
  $$
  
  \begin{equation}\label{pvBI3}
     \times {}_{2}F_{1}^{\beta}((m-1)\beta/2+1,-a+(m-1)\beta/2+1;b+(m-1)\beta/2+1;-(\mathbf{I}_{m} - \mathbf{\Omega})\mathbf{\Omega}^{-1}).
  \end{equation}
 
\end{corollary}
\begin{proof}
  The results follow from theorems \ref{teo1} and \ref{teo3} by the elemental properties
  $$
    P(\mathbf{U} > \mathbf{\Omega}) = P(-\mathbf{U} < -\mathbf{\Omega}) = P(\mathbf{I}_{m}-\mathbf{U} < \mathbf{I}_{m}-\mathbf{\Omega}).
  $$
\end{proof}

\subsection{Matricvariate beta type II distribution}

\begin{theorem}\label{teo4}
  Let $\mathbf{F} \sim \mathcal{F}_{m}^{\beta}(a,b)$ then its lower probability is 
  $$
    P(\mathbf{F} < \mathbf{\nabla}) = \frac{\mathcal{B}_{m}^{\beta}[a, (m-1)\beta/2 +1]}{\mathcal{B}_{m}^{\beta}[a,b]} |\mathbf{\nabla}|^{a} \hspace{5cm}
  $$
  $$
    \hspace{3cm}
    \times {}_{2}F_{1}^{\beta}(a+b,a;a+(m-1)\beta/2+1;-\mathbf{\nabla}).
  $$
  Alternatively, with the Euler relation (\ref{k1}), we obtain
  $$
    P(\mathbf{F} < \mathbf{\nabla}) = \frac{\mathcal{B}_{m}^{\beta}[a, (m-1)\beta/2 +1]}{\mathcal{B}_{m}^{\beta}[a,b]} |\mathbf{\nabla}|^{a} |\mathbf{I}_{m}+\mathbf{\nabla}|^{-(a+b-(m-1)\beta/2 -1)} \hspace{4cm}
  $$
  $$
    \hspace{2cm}
    \times {}_{2}F_{1}^{\beta}(-(b-(m-1)\beta/2-1),(m-1)\beta/2+1;a+(m-1)\beta/2+1;-\mathbf{\nabla}).
  $$
  Also, observing that $|\mathbf{\nabla}|^{a}|\mathbf{I}_{m}+\mathbf{\nabla}|^{-a} = |\mathbf{I}_{m}+\mathbf{\nabla}^{-1}|^{a}$ and by Euler relation (\ref{k2}) we have
  $$
    P(\mathbf{F} < \mathbf{\nabla}) = \frac{\mathcal{B}_{m}^{\beta}[a, (m-1)\beta/2 +1]}{\mathcal{B}_{m}^{\beta}[a,b]}  |\mathbf{I}_{m}+\mathbf{\nabla}^{-1}|^{-a}\hspace{5cm}
  $$
  $$
    \hspace{1cm}
    \times {}_{2}F_{1}^{\beta}\left(-(b-(m-1)\beta/2-1),a;a+(m-1)\beta/2+1;\left(\mathbf{I}_{m}+\mathbf{\nabla}^{-1}\right)^{-1}\right).
  $$
\end{theorem}
\begin{proof}
  The distribution function of $\mathbf{F}$ is written as
  $$
    P(\mathbf{F} < \mathbf{\nabla}) = \frac{1}{\mathcal{B}_{m}^{\beta}[a,b]}\int_{\mathbf{0}<\mathbf{F}<\mathbf{\nabla}}|\mathbf{F}|^{a-(m-1)\beta/2-1} |\mathbf{I}_{m}+\mathbf{F}|^{-(a+b)}(d\mathbf{F}).
  $$
  Define $\mathbf{F} = \mathbf{\nabla}^{1/2}\mathbf{R}\mathbf{\nabla}^{1/2}$, where $\mathbf{\nabla}^{1/2}$ is such that $\left(\mathbf{\nabla}^{1/2}\right)^{2} = \mathbf{\nabla}$. Then by Proposition \ref{prop1} i) we have that $(d\mathbf{F}) = |\mathbf{\nabla}|^{(m-1)\beta/2 +1}(d\mathbf{R})$. Also observe that $\mathbf{0}<\mathbf{R}<\mathbf{I}_{m}$, therefore
  $$
    P(\mathbf{F} < \mathbf{\nabla}) = \frac{|\mathbf{\nabla}|^{a}}{\mathcal{B}_{m}^{\beta}[a,b]} \int_{\mathbf{0}<\mathbf{R}<\mathbf{I}_{m}} |\mathbf{R}|^{a-(m-1)\beta/2-1} |\mathbf{I}_{m}+\mathbf{\nabla R}|^{-(a+b)}(d\mathbf{R}).
  $$
  Recall that 
  $$
    |\mathbf{I}_{m}-\mathbf{X}|^{-a} = {}_{1}F_{0}^{\beta}(a;\mathbf{X}) = \sum_{k=0}^{\infty}\sum_{\kappa}\frac{[a]_{\kappa}^{\beta}}{k!} C_{\kappa}^{\beta}(\mathbf{X}),
  $$
  where $ C_{\kappa}^{\beta}(\mathbf{X})$ denotes the Jack polynomials (also termed zonal polynomials), see \citet{S97}.
  Hence
  $$
    P(\mathbf{F} < \mathbf{\nabla}) = \frac{|\mathbf{\nabla}|^{a}}{\mathcal{B}_{m}^{\beta}[a,b]}  \sum_{k=0}^{\infty}\sum_{\kappa} \frac{[a+b]_{\kappa}^{\beta}}{k!} \int_{\mathbf{0}<\mathbf{R}<\mathbf{I}_{m}}\mathbf{R}|^{a-(m-1)\beta/2-1}C_{\kappa}^{\beta}(\mathbf{-\nabla R})(d\mathbf{R}).
  $$
  From \citet[Equation 3.30, p. 102 and Equation 2.4, p. 91]{dg:14} the desired result is obtained.
\end{proof}
\begin{corollary}\label{cor2}
  Suppose that $\mathbf{F} \sim \mathcal{F}_{m}^{\beta}(a,b)$ then 
  $$
    P(\mathbf{F} > \mathbf{\nabla}) = \frac{\mathcal{B}_{m}^{\beta}[b, (m-1)\beta/2 +1]}{\mathcal{B}_{m}^{\beta}[a,b]} |\mathbf{\nabla}|^{-b} \hspace{5cm}
  $$
  \begin{equation}\label{pvBII1}
    \hspace{3cm}
    \times {}_{2}F_{1}^{\beta}(a+b,b;b+(m-1)\beta/2+1;-\mathbf{\nabla}^{-1}).
  \end{equation}
  Or 
  $$
    P(\mathbf{F} > \mathbf{\nabla}) = \frac{\mathcal{B}_{m}^{\beta}[b, (m-1)\beta/2 +1]}{\mathcal{B}_{m}^{\beta}[a,b]} |\mathbf{\nabla}|^{-b} |\mathbf{I}_{m}+\mathbf{\nabla}^{-1}|^{-(a+b-(m-1)\beta/2 -1)} \hspace{4cm}
  $$
  \begin{equation}\label{pvBII2}
    \hspace{1cm}
    \times {}_{2}F_{1}^{\beta}(-(a-(m-1)\beta/2-1),(m-1)\beta/2+1;b+(m-1)\beta/2+1;-\mathbf{\nabla}^{-1}).
  \end{equation}
  And
  $$
    P(\mathbf{F} > \mathbf{\nabla}) = \frac{\mathcal{B}_{m}^{\beta}[b, (m-1)\beta/2 +1]}{\mathcal{B}_{m}^{\beta}[a,b]}  |\mathbf{I}_{m}+\mathbf{\nabla}|^{-b}\hspace{5cm}
  $$
  \begin{equation}\label{pvBII3}
    \times {}_{2}F_{1}^{\beta}\left(-(a-(m-1)\beta/2-1),b;b+(m-1)\beta/2+1;(\mathbf{I}_{m}+\mathbf{\nabla})^{-1}\right).
  \end{equation}
\end{corollary}
\begin{proof}
Observing that
$$
   P(\mathbf{F} > \mathbf{\nabla}) =  P(\mathbf{F}^{-1} < \mathbf{\nabla}^{-1})
$$
the proof is a consequence of theorems \ref{teo2} and \ref{teo4}.
\end{proof}

\subsection{Computation}
This work and similar research of the authors are inscribed into a very profound problem of matrix variate distribution theory related with a feasible computation. The foundations of the MANOVA probability setting in this paper start in the 50's around the difficult problem of finding the joint density function of the latent roots of an $\mathbf{X} \in \mathfrak{S}_{m}^{\beta}$ with probability density function $f(\mathbf{X})$. At that time, the theory of real normed division algebras did not exists, then the real case ($\beta=1$) came first and then without any relation the complex case ($\beta=2$) was independently and hardly constructed.  In the real case, the addressed joint distribution requires the following integral over the invariant normalised Haar probability measure  $(d\mathbf{H})$:
$$
   \int_{O(m)}f(\mathbf{H}\mathbf{L}\mathbf{H}')(d\mathbf{H}).
$$
Even in the Gaussian central kernel $f$, the solution demanded the creation in \citet{j:60} of the so called zonal or James polynomials of one matrix argument $C_{\kappa}^{1}(\mathbf{X})$. A.T. James considered a number methods for computing the polynomials, but the most efficient technique arrived with \citet{j:68} by establishing a crucial recurrent method via the Laplace-Beltrami operator. But the task for computation of low order polynomials was so extreme that a Ph. thesis was needed for constructing the polynomials up degree 12th (\citet{pj:74}). The integral and related works for the central complex case required a parallel theory and took much time. It started in \citet{j:64}, and a similar work to Parkust thesis without a recurrence method was performed by F. Caro-Lopera for obtaining the $C_{\kappa}^{1}(\mathbf{X})$ and some related functions, see \citet{cn:06}, \citet{gncl:06}, \citet{gncl:05}. The construction of the complex zonal polynomials  by the Laplace-Beltrami operator appeared later in \citet{dgcl:07}. As in the complex real case (\citet{j:68}), a separatelly trial of getting an exact formulae only arrived for the second order in \citet{cldg:06}. From a numerical point of view the computation of hypergeometric functions, which are series of zonal polynomials, was given by \citet{ke:06}, then all the works since 60's about the central matrix variate theory via James polynomials were numerically approximated. The addressed work was also set for real normed division algebras based on the so called Jack polynomials by \citet{S97} $C_{\kappa}^{1}(\mathbf{X})$ which includes the real and complex zonal polynomials, but also the new quaternions and octonions. Exact formulae for the Jack polynomials are so elusive, in fact, only the second order case has been solve by  It should be noted that \citet{S97} and \citet{ke:06} are prescribed for the definite positive case, the unified positive and semidefinite positive real setting was given by \citet{dgcl:06}. Applications of the semidefinite positive approach are still to research.
Now, the central case for the addressed general kernel $f(\cdot)$, in order to obtain the joint distribution of the latent roots of an elliptically contoured distributions, was given by \citet{clgb:14} in terms of computable series of $C_{\kappa}^{1}(\mathbf{X})$. The generalization to real normed division algebras, with the corresponding computable series involving $C_{\kappa}^{\beta}(\mathbf{X})$, was provided by  \citet{dg:13}. 

Finally, the main problem of computation for possible extensions of the probabilities derived here arrives in the non central real case. It forces the apparition in \citet{d:79} and \citet{d:80} of the termed Davis or invariant polynomials of matrix several matrix arguments $C_{\phi}^{\kappa[r],\beta}$, $\beta=1$, extending the real zonal polynomials of one matrix argument $C_{\kappa}^{1}(\mathbf{X})$, see \citet{dg:14} for general case. 
\citet{d:79} and \citet{d:06} maintained the conjecture that they could be obtained in a recurrent way as the zonal polynomials, however \citet{cl:16} proved that impossibility. Then, until now, it has left the problem of computation of dozens of papers without a plausible computation.
Fortunately, all the probabilities here derived involves zonal polynomials, which are easily computable by using the approximation of \citet{ke:06}.

\section{New tests on matrix variate distributions}\label{sec:6} 

The expected matrix $p$-value arises naturally in this section by providing a new approach for testing the general multivariate linear hypothesis. To motivate the result, we review most of the statistical literature approaches. The test statistics for all the known criteria are showed in  tables \ref{tab1} and \ref{tab2}. 

\begin{table}[!htp]
 \centering
   \caption{\scriptsize Criteria for testing the null hypothesis}\label{tab1}
\medskip   
\begin{tiny}
\begin{minipage}[t]{360pt}
\begin{tabular}{||l|l|l||} \hline\hline
    &&\\
    Criterion & Statistics & References\\ && \\ \hline\hline
    &&\\
    Wilks's $\Lambda$ \footnote{\tiny The decision rule for all the criteria is: \textbf{reject $H_{0}$ if the statistic $\geq$ critical value.}
                         However, for Wilks's $\Lambda$ and Pillai's $W^{(s)}$ criteria, the decision rule is (this class of test are known in statistical literature as \textbf{inverse test}, see \citet[p. 162]{re:02}): \textbf{reject $H_{0}$ if the statistic $\leq$ critical value.}} 
                         & $\begin{array}{ccl}
                           \Lambda &=& \displaystyle\frac{|\mathbf{S}_{E}|}{|\mathbf{S}_{E} + \mathbf{S}_{H}|} \\
                             &=& \displaystyle\prod_{i = 1}^{s} \frac{1}{1 + \lambda_{i}} \\
                             &=& \displaystyle\prod_{i=1}^{s}(1-\theta_{i}). 
                         \end{array}
                         $ 
                         & \begin{tabular}{l}
                             see \citet{w:32}, \\
                             \citet[p. 161]{re:02} \\
                             and \citet[p. 5 and pp. 14-51]{k:83}.
                           \end{tabular} \\ &&\\ \hline
  &&\\                           
    \begin{tabular}{l}
      Wilks's $U$ \\
      Gnanadesikan's $U$ 
    \end{tabular}        & $\begin{array}{ccl}
                           U &=& \displaystyle\frac{|\mathbf{S}_{H}|}{|\mathbf{S}_{E} + \mathbf{S}_{H}|} \\
                             &=& \displaystyle\prod_{i=1}^{s}\frac{\lambda_{i}}{1 + \lambda_{i}}\\
                             &=& \displaystyle\prod_{i =1}^{s} \theta_{i}. 
                         \end{array}
                         $ 
                         & \begin{tabular}{l}
                             \citet[p. 72]{rgs:71}, \\
                             \citet[p. 413]{s:84},\\
                             and \citet[p. 6]{k:83}.
                           \end{tabular}  \\ &&\\ \hline
  &&\\                           
  \begin{tabular}{l}
      Wilks's $V$ \\
      Olson's $V$ 
    \end{tabular}        & $\begin{array}{ccl}
                           V &=& \displaystyle\frac{|\mathbf{S}_{H}|}{|\mathbf{S}_{E}|} \\
                             &=& \displaystyle\prod_{i=1}^{s}\lambda_{i}\\
                             &=& \displaystyle\prod_{i =1}^{s} \frac{\theta_{i}}{(1 - \theta_{i})}. 
                         \end{array}
                         $ 
                         & \begin{tabular}{l}
                             \citet{w:32}, \\
                             \citet{ol:74},\\
                             \citet[p. 8]{k:83},\\
                             and \citet{dgcl:08} 
                           \end{tabular}  \\ &&\\ \hline
  &&\\                           
  Lawley-Hotelling's $U^{(s)}$  & $\begin{array}{ccl}
                           U^{(s)} &=& \tr (\mathbf{S}_{E}^{-1} \mathbf{S}_{H}) \\
                             &=& \displaystyle\sum_{i=1}^{s}\lambda_{i} \\
                             &=& \displaystyle\sum_{i=1}^{s} \frac{\theta_{i}}{(1 - \theta_{i})}.  
                         \end{array}
                         $                         
                         & \begin{tabular}{l}
                             see \citet[p. 466]{mh:05}, \\
                             \citet[p. 167]{re:02} \\
                             and \citet[p. 6 and pp. 118-135]{k:83}.
                           \end{tabular}  \\ &&\\ \hline
  &&\\                            
  Pillai's $V^{(s)}$  & $\begin{array}{ccl}
                           V^{(s)} &=& \tr ((\mathbf{S}_{E}+ \mathbf{S}_{H})^{-1}\mathbf{S}_{H}) \\
                             &=& \displaystyle\sum_{i=1}^{s}\frac{\lambda_{i}}{(1+\lambda_{i})} \\
                             &=& \displaystyle\sum_{i=1}^{s}\theta_{i}.  
                         \end{array}
                         $                                                  
                         & \begin{tabular}{l}
                             see \citet[p. 466]{mh:05}, \\
                             \citet[p. 168]{re:02} \\
                             and \citet[p. 6 and pp. 136-153]{k:83}.
                           \end{tabular}  \\ &&\\ \hline
  &&\\                           
  Pillai's $W^{(s)}$  & $\begin{array}{ccl}
                           W^{(s)} &=&  \tr ((\mathbf{S}_{E}+ \mathbf{S}_{H})^{-1}\mathbf{S}_{E}) \\
                             &=& \displaystyle\sum_{i=1}^{s}\frac{1}{(1+\lambda_{i})} =\sum_{i=1}^{s}(1-\theta_{i}) \\
                             &=& (1 - V^{(s)}/s).  
                         \end{array}
                         $                               
                         & \citet{p:55}.  \\ &&\\ \hline
  &&\\                            
  Pillai's $H^{(s)}$  & $\begin{array}{ccl}
                           H^{(s)} &=&  \displaystyle\frac{s}{\displaystyle\sum_{i=1}^{s}(1+\lambda_{i})} \\
                             &=& s\left \{\displaystyle\sum_{i=1}^{s} (1-\theta_{i})^{-1}\right \}^{-1} \\
                             &=& (1 + U^{(s)}/s)^{-1}.  
                         \end{array}
                         $                                                       
                         & \begin{tabular}{l}
                             see \citet{p:55}, \\
                             and \citet[p. 8]{k:83}.
                           \end{tabular}  \\ &&\\                            
  \hline\hline
\end{tabular}
\end{minipage}
\end{tiny}
\end{table}

\begin{table}[!htp]
 \centering
   \caption{\scriptsize Continuation...}\label{tab2} 
\medskip   
\begin{tiny}
\begin{minipage}[t]{360pt}
\begin{tabular}{||l|l|l||} \hline\hline
    &&\\
    Criteria \footnote{\tiny The tables for critical values of all the criteria are tabulated in terms of the parameters $(m,\nu_{H},\nu_{E})$ or in terms of the parameters $(s,n,h)$,  where 
     $$
       s = \min(m,\nu_{H}), \ \ n = (|\nu_{H} - m|-1)/2 \ \mbox{ and } \ h = (\nu_{E}-m-1)/2.
     $$
     In general, the tables have been computed by assuming that $m \leq \nu_{H}$ and $m\leq \nu_{E}$. If $m > \nu_{H}$ then use the combination of parameters $(\nu_{H}, m, \nu_{E}+\nu_{H}-m)$ in place of $(m,\nu_{H},\nu_{E})$,   see \citet[eq. (7), p. 455]{mh:05}, \citet[p. 96]{sk:79} or \citet[p. 167]{re:02}.} 
    & Statistics & References\\ && \\ \hline\hline
    &&\\
  Pillai's $R^{(s)}$ \footnote{\tiny Where $U^{'(s)}$ is the same $U^{(s)}$ but with $m$ and $h$ interchanged.} & $\begin{array}{ccl}
                           R^{(s)} &=&  \displaystyle\frac{s}{\displaystyle\sum_{i=1}^{s}\frac{1 + \lambda_{i}}{\lambda_{i}}}\\
                             &=& \displaystyle s\left \{\displaystyle\sum_{i=1}^{s}\theta_{i}^{-1}\right \}^{-1} \\
                             &=&  (1+U^{'(s)}/s)^{-1}. 
                         \end{array}
                         $                                             
                         & \begin{tabular}{l}
                             see \citet{p:55}, \\
                             and \citet[p. 8]{k:83}.
                           \end{tabular} \\ &&\\ \hline
  Pillai's $T^{(s)}$  & $\begin{array}{ccl}
                           T^{(s)} &=&  \displaystyle s\left \{\sum_{i=1}^{s} \lambda_{i}^{-1}\right \}^{-1}\\
                             &=& \displaystyle \frac{s}{\displaystyle\sum_{i=1}^{s} \frac{1-\theta_{i}}{\theta_{i}}} \\
                             &=&   \displaystyle\frac{R^{(s)}}{1-R^{(s)}}.
                         \end{array}
                         $                                 
                         & \begin{tabular}{l}
                             see \citet{p:55}, \\
                             and \citet[p. 8]{k:83}.
                           \end{tabular}  \\ &&\\ \hline
  &&\\                                                 
    Roy's $\lambda_{\mbox{max}}$ &  $\begin{array}{ccl}
                           \lambda_{\mbox{max}} &=&  \lambda_{\mbox{max}}(\mathbf{S}_{E}^{-1} \mathbf{S}_{H})\\[2ex]
                             &=& \displaystyle \frac{\theta_{\mbox{max}}}{1-\theta_{\mbox{max}}}.
                         \end{array}
                         $
                         & \begin{tabular}{l}
                             see \citet{r:57}, \\
                             and \citet[p. 7 and pp. 62-86]{k:83}.
                           \end{tabular}  \\ &&\\ \hline
  &&\\                           
    Roy's $\theta_{\mbox{max}}$ & $\begin{array}{ccl}
                           \theta_{\mbox{max}} &=&  \theta_{\mbox{max}}(\mathbf{S}_{E}+ \mathbf{S}_{H})^{-1}\mathbf{S}_{H})\\[2ex]
                             &=& \displaystyle \frac{\lambda_{\mbox{max}}}{1+\lambda_{\mbox{max}}}.
                         \end{array}
                         $ 
                         & \begin{tabular}{l}
                             see \citet{r:57}, \\
                             \citet[p. 481]{mh:05},\\
                             and p. 7, pp. 52-61, 87-104 and 105-117\\
                             in \citet{k:83}.
                           \end{tabular}  \\ &&\\ \hline
  &&\\                                                      
    Anderson's $\lambda_{\mbox{min}}$ &  $\begin{array}{ccl}
                           \lambda_{\mbox{min}} &=&  \lambda_{\mbox{min}}(\mathbf{S}_{E}^{-1} \mathbf{S}_{H})\\[2ex]
                             &=& \displaystyle \frac{\theta_{\mbox{min}}}{1-\theta_{\mbox{min}}}..
                         \end{array}
                         $
                         & \begin{tabular}{l}
                             see \citet{r:57},\\
                             \citet{a:82},\\
                             and \citet[p. 7]{k:83}.
                           \end{tabular}  \\ &&\\ \hline
  &&\\                           
    Roy's $\theta_{\mbox{min}}$ & $\begin{array}{ccl}
                           \theta_{\mbox{min}} &=&  \theta_{\mbox{min}}(\mathbf{S}_{E}+ \mathbf{S}_{H})^{-1}\mathbf{S}_{H})\\[2ex]
                             &=& \displaystyle \frac{\lambda_{\mbox{min}}}{1+\lambda_{\mbox{min}}}.
                         \end{array}
                         $
                         & \begin{tabular}{l}
                             see \citet{p:55},\\
                             \citet{n:48},\\
                             and \citet{r:57}.
                           \end{tabular}  \\ &&\\ \hline
  &&\\                           
    Dempster's $T_{D}$ & $T_{D} = (\tr \mathbf{S}_{H})/(\tr \mathbf{S}_{E}),$ 
                        & \begin{tabular}{l}
                             see \citet{d:58},\\
                             \citet{d:60},\\
                             and \citet{fhw:04}.
                           \end{tabular}  \\ &&\\                            
\hline\hline
\end{tabular}
\end{minipage}
\end{tiny}
\end{table}

Now, in terms of the general linear hypothesis, in the univariate case ($m = 1$), we have that \footnote{Remember that this decision rule is obtained via the generalised likelihood ratio test, see \citet[Definition 2.8.4, p. 85 and p. 185]{g:76}.} 
$$
  \mbox{reject }H_{0}:\mathbf{C}\mathbb{B}=\mathbf{h}  \mbox{ if } F_{c}\geq F_{t},
$$
where $F_{c}$ is termed $F$-calculated, and is given by
$$ 
  F_{c} = \frac{\displaystyle\frac{SSH}{\nu_{H}}}{\displaystyle\frac{SSE}{\nu_{E}}}
$$
$SSH$ denotes the sum of squares due to the hypothesis and $SSE$ denotes the sum of squares due to the error. Here $F_{t}\equiv F_{\alpha, \nu_{H}, \nu_{E}}$ is the upper $\alpha$ probability point of the F-distribution with $\nu_{H}$ and $\nu_{E}$ degrees of freedom. Alternatively $H_{0}$ is rejected if $P(F >F_{c}) \equiv \mathbf{p}-$\textbf{value} is less than a certain preset value (usually 0.05 or 0.01).

Extension of rejecting the null hypothesis $H_{0}:\mathbf{C}\mathbb{B}\mathbf{M}=\mathbf{H}$ into the multivariate general hypothesis setting, arise naturally if we \textbf{\textit{heuristically}} set the decision rule as 
$$
  \mbox{reject }H_{0}:\mathbf{C}\mathbb{B}\mathbf{M}=\mathbf{H}  \mbox{ if } \mathbf{F}_{c} > \mathbf{F}_{t},
$$
and now $\mathbf{F}_{c} = \mathbf{S}_{E}^{-1}\mathbf{S}_{H}$ or the symmetric form is taken:  
$$
  \mathbf{F}_{c} = \mathbf{S}_{E}^{-1/2}\mathbf{S}_{H}\mathbf{S}_{E}^{-1/2} = \mathbf{S}_{H}^{1/2}\mathbf{S}_{E}^{-1}\mathbf{S}_{H}^{1/2}.
$$
Under null hypothesis, $\mathbf{F}_{c}$ has a beta type II distribution with $\nu_{H}$ and $\nu_{E}$ parameters, i.e. $\mathbf{F}_{c} \sim \mathcal{F}_{g}^{\beta}(\nu_{H},\nu_{E})$, see \citet[Theorem 10.4.1, p. 449]{mh:05} and \citet{j:64}. Or
$$
  \mathbf{U}_{c} = \left(\mathbf{S}_{H} + \mathbf{S}_{E}\right)^{-1/2}\mathbf{S}_{H}\left(\mathbf{S}_{H} + \mathbf{S}_{E}\right)^{-1/2} = \mathbf{S}_{H}^{1/2}\left(\mathbf{S}_{H} + \mathbf{S}_{E}\right)^{-1}\mathbf{S}_{H}^{1/2},
$$
where $\mathbf{U}_{c} \sim \mathfrak{B}_{g}^{\beta}(\nu_{H},\nu_{E})$, namely, under the null hypothesis, we obtain a beta type I distribution  with parameters $\nu_{H}$ and $\nu_{E}$. Recall also that, $\mathbf{U}_{c} = \mathbf{I}_{g}-(\mathbf{I}_{g}+ \mathbf{F}_{c})^{-1}$ and $\mathbf{F}_{c} = (\mathbf{I}_{g} - \mathbf{U}_{c})^{-1}- \mathbf{I}_{g}$, see \citet{sk:79}. 

The new approach just requires some insights about the explanation of the measure, via $p$-value, of the well known Loewner order, a plausible task which can be heuristically explained in the referred statement of \citet{p:55}.

Tables \ref{tab1} and \ref{tab2} just proposes different metrics to discern when $\mathbf{F}_{c} > \mathbf{F}_{t}$. If $\rho(\cdot)$ denotes a metric, then $\mathbf{F}_{c} > \mathbf{F}_{t}$ implies that  $\rho(\mathbf{F}_{c}) > \rho(\mathbf{F}_{t})$, but not the opposite.  The references consider the determinant, the trace, the maximum and minimum eigenvalue of the matrices $\mathbf{F}_{c}$ and $\mathbf{U}_{c}$. Dempster, for example emulates a quotient of the traces of $\mathbf{S}_{H}$ and $\mathbf{S}_{E}$ in analogy to the quotient of determinants proposed by Wilks. Alternatively, when such criteria are written in terms of the eigenvalues of $\mathbf{F}_{c}$ and $\mathbf{U}_{c}$ a number of new metrics raise. Writing the trace in terms of the eigenvalues, a proportional quantity to the eigenvalue arithmetic mean appears, this motivates metrics based on geometric mean or harmonic mean of the eigenvalues, see \citet{p:55}, etc.

Now we are in a position of proposing our new heuristic approach for the decision rule of the general multivariate linear hypothesis test.

\medskip
\begin{center}
\colorbox[gray]{0.8}{$
   \mbox{Reject }H_{0}:\mathbf{C}\mathbb{B}\mathbf{M}=\mathbf{H}  \mbox{ if } P(\mathbf{F} >\mathbf{F}_{c}) \equiv \mbox{ p-value}  < \alpha,
$}\\[1ex]
\end{center}
where the $p$-value is reached by corollaries \ref{cor1} or \ref{cor2}. As usual, $p<0.05$ and $p<0.01$ are typically considered as statistically significant and highly significant, respectively.
\medskip
\begin{remark}
By Theorem 5.3.1 in \citet[p. 182]{gv:93}, all the distribution functions and upper probability functions here derived are invariant under the family of elliptically contoured distributions. Thus, they coincide with the distributions under the normality assumption, see also\citet{fz:90}.
\end{remark}

\section{Applications}

For validation of our theory, we develop two classical examples: a multivariate one-way analysis of variance model and a balanced multivariate two-way fixed-effects analysis of variance model.

\begin{example}
Let us consider of \citet[Example 6.1.7, p. 171]{re:02} about the comparison of apples threes with 6 different rootstocks. The data arrived in the context of an experiment back to 1918-1934, where the following variables were registered: trunk girth at 4 and 15 years, in $mm \times 100$; extension growth at 4 years, in $m$; and, weight of tree above ground at 15 years, in $lb \times 1000$. From \citet[p. 170]{re:02} the symmetric version of $\mathbf{F}_{c} = \mathbf{E}^{-1/2}\mathbf{H}\mathbf{E}^{-1/2}$ is
$$
  \mathbf{F}_{c} = \begin{pmatrix*}[r]
             0.05322776  & -0.01487401 &  0.1982486 & 0.07238464\\
            -0.01487401  &  0.38103449 & -0.3317237 & 0.09765930\\
             0.19824861  & -0.33172370 &  1.6121905 & 0.42164487\\
             0.07238464  &  0.09765930 &  0.4216449 & 0.87498679
       \end{pmatrix*},
$$
with eigenvalues
$$
  (\lambda_{1},\lambda_{2},\lambda_{3},\lambda_{4}) = (1.875848, 0.7906445, 0.2289795, 0.02596715)
$$
 
Corollaries \ref{cor1} and \ref{cor2} provide the following rule of decision:
\begin{center}
   {\bf reject the null hypothesis by $\mathbf{p}$-value = 8.679157e-18.}
\end{center}
This decision coincides with the criteria of Wilks, Pillai, Lawley-Hotelling and Roy, calculated in \citet[pp. 172-173]{re:02}.

\begin{remark}
   For the application of the test criterion, we study carefully the corresponding hypergeometric function with one matrix argument:
   \begin{enumerate}
     \item Here the function (\ref{pvBII1}) depends on the argument $- \mathbf{\nabla}^{-1}$, where $||- \mathbf{\nabla}^{-1}||= 38.5102$. Therefore the corresponding hypergeometric series diverges.
     \item However, if any of the parameters $a_{1}, \dots,a_{p}$ is zero, the hypergeometric function terminates and sums 1.
     \item For practical purposes, the presence of a negative integer in the parameters $a_{1}, \dots,a_{p}$ forces the hypergeometric series to be a polynomial of degree $nm$, where $n = - a_{i}$ for some $i = 1, \dots, p$.
     \item In this example the function (\ref{pvBI2}) does not involve a negative integer parameter, then it converges slowly. This computational issue is addressed by \citet{ke:06}, as: "Several problems remain open, among them automatic detection of convergence, ..., and the best way to truncate the series". However, all the examples here considered reach convergence with a sufficient truncation.
     \end{enumerate}
\end{remark}
\end{example}

\begin{example}
Now, we study the randomised complete design with factorial arrangement $2\times4$ given in \citet[Example 6.5.2, p. 191]{re:02}.  The data appear in \citet[Table 6.6, p. 192]{re:02}. The experiment involved a $2 \times 4$ design with $4$ replications, for a total of 32 observation vectors. The factors were rotational velocity [$\mathbf{A}_{1}$ (fast) and $\mathbf{A}_{2}$ (slow)] and lubricants [four types, $\mathbf{B}_{1}, \mathbf{B}_{2}, \mathbf{B}_{3}$ and $\mathbf{B}_{4}$]. The experimental units were 32 homogeneous pieces of bar steel. Two variables were measured on each piece of bar steel: $y_{1}$ = ultimate torque and $y_{2}$ = ultimate strain. From \citet[p. 193]{re:02} the symmetric versions of $\mathbf{F}_{c\mathbf{A}}= \mathbf{E}^{-1/2}\mathbf{H_{A}}\mathbf{E}^{-1/2}$, $\mathbf{F}_{c\mathbf{B}} = \mathbf{E}^{-1/2}\mathbf{H_{B}}\mathbf{E}^{-1/2}$  and $\mathbf{F}_{c\mathbf{AB}} = \mathbf{E}^{-1/2}\mathbf{H_{AB}}\mathbf{E}^{-1/2}$ and their corresponding eigenvalues are
$$
  \mathbf{F}_{c\mathbf{A}} = \begin{pmatrix*}[r]
             0.273464 & 0.478255\\
             0.478255 & 0.836411
       \end{pmatrix*}, \quad (\lambda_{1\mathbf{A}}, \lambda_{2\mathbf{A}}) =(1.109875, 0.000000),
$$
$$
  \mathbf{F}_{c\mathbf{B}} = \begin{pmatrix*}[r]
            0.336837 & -0.160550\\
           -0.160550 &  0.100913
       \end{pmatrix*}, \quad (\lambda_{1\mathbf{B}}, \lambda_{2\mathbf{B}}) =(0.418102, 0.019648),
$$
and 
$$
  \mathbf{F}_{c\mathbf{AB}} = \begin{pmatrix*}[r]
            0.028637 & 0.027744\\
            0.027744 & 0.043918
       \end{pmatrix*}, \quad (\lambda_{1\mathbf{AB}}, \lambda_{2\mathbf{AB}}) =(0.065054, 0.007501),
$$
respectively.

To calculate the $p$-value for the factor $\mathbf{A}$, note that the rank of $\mathbf{F}_{c\mathbf{A}} = 1 < m = 2$; then the parameter substitutions of Table \ref{tab2} are needed. Then we obtain the next decision: 
\begin{center}
   {\bf reject the null hypothesis by $\mathbf{p}$-value = 2.765e-05.}
\end{center}
For factor $\mathbf{A}$ the $p$-value was evaluated using all expressions (\ref{pvBI1}) to (\ref{pvBII3}) and for the correctness of our theory, the same result was obtained.

The $p$-value for the hypothesis of factor $\mathbf{B}$ is obtained under expressions (\ref{pvBI1}) to (\ref{pvBII3}). In this case, the three parameters $a_{1}$, $a_{2}$ and $b_{1}$ are positive integers or fractions and expression (\ref{pvBII1}) diverges because $||-\mathbf{F}_{c\mathbf{B}}^{-1}|| = 50.894747 > 1$. Evaluation of the $p$-value via (\ref{pvBII2}) gives $||-\mathbf{F}_{c\mathbf{B}}^{-1}|| = 50.894747 > 1$, but in this case $a_{1} = 0$, and the series sums 1. Hence, the rule decision is: 
\begin{center}
   {\bf reject the null hypothesis by $\mathbf{p}$-value = 0.0119703.}
\end{center}

In this test, (\ref{pvBI2}) required larger truncation, since $a_{1},a_{2}$ and $b_{1}$ are positive fractions.

\medskip

Finally, for the $\mathbf{AB}$ interaction testing, (\ref{pvBII1}) diverges since $||-\mathbf{F}_{c\mathbf{AB}}^{-1}|| = 133.31874 > 1$ and  $a_{1}$, $a_{2}$ and $b_{1}$ are positive fractions. A similar situation occurs with (\ref{pvBII2}), since $||-\mathbf{F}_{c\mathbf{AB}}^{-1}|| = 133.31874 > 1$, but in this case the hypergeometric series sums 1, since $a_{1} =0$. The remaining expressions for the $p$-value converge. The decision rule is 
\begin{center}
   {\bf do not reject the null hypothesis by $\mathbf{p}$-value = 0.4291338.}
\end{center}

According to \citet[p. 194]{re:02} only the factor $\mathbf{A}$ has a highly significant effect under the Wilks's criterion. However, under Roy's criteria, the conclusions coincide exactly with those obtained in this article, that is: factor $\mathbf{A}$ has a highly significant effect, factor $\mathbf{B}$ has a significant effect, and factor $\mathbf{AB}$ has no significant effect. For completeness, Table \ref{table3} presents the MANOVA obtained with the R program version 4.3.3, \citet{r:24}. 

\begin{normalsize}
\begin{table}[!htp]
  \caption{MANOVA with Roy's Criterion}\label{table3}
  \centering
   \begin{tabular}{||c|c|c|c|c|c|l||}
     \hline\hline
      & Df & Roy & approx F & num Df & den Df & $Pr(>F)$ \\
      \hline\hline
     A & 1 & 1.10988 & 12.7636 & 2 & 23 & 0.0001867$^{***}$ \\
     B & 3 & 0.41810 & 3.3448 & 3 & 24 & 0.0359080$^{*}$ \\
     AB & 3 & 0.06505 & 0.5204 & 3 & 24 & 0.6733819 \\
     Residuals & 24 &   &   &   &   &   \\
     \hline\hline
   \end{tabular}
\end{table}
\end{normalsize}
\end{example}

Our criterion can be applied in several multivariate hypothesis testing situations. 

The two-sample univariate hypothesis $H_{0}: \sigma_{1}^{2}= \sigma_{2}^{2}$ versus $H_{1}: \sigma_{1}^{2} \neq \sigma_{2}^{2}$ is tested by computing
$$
  F = \frac{s_{1}^{2}}{s_{2}^{2}},
$$
where $s_{1}^{2}$ and $s_{2}^{2}$ are the variances of the two samples. Under $H_{0}$, and assuming normality, $F$ is distributed as $F(\nu_{1},\nu_{2})$, where $\nu_{1}$ and $\nu_{2}$ are the degrees of freedom of $s_{1}^{2}$ and $s_{2}^{2}$ (typically, $n_{1}-1$ and $n_{2}-1$). Note that $s_{1}^{2}$ and $s_{2}^{2}$ must be independent, which shall hold if the two samples are independent, see \citet[pp. 254-255]{re:02}. The rule of decision is:
\begin{center}
{\bf reject the null hypothesis $H_{0}$ if  $F > F_{t}$},
\end{center}
where $F_{t}\equiv F_{\alpha, \nu_{1}, \nu_{2}}$ is the upper $\alpha$ probability point of the F-distribution with $\nu_{1}$ and $\nu_{2}$ degrees of freedom.

Now, we propose the following multivariate version of equality of variances in terms of our test criterion.

For, two-sample multivariate hypothesis $H_{0}: \mathbf{\Sigma}_{1}= \mathbf{\Sigma}_{2}$ versus $H_{1}: \mathbf{\Sigma}_{1} \neq \mathbf{\Sigma}_{2}$, the following decision rule is proposed\medskip
\begin{center}
\colorbox[gray]{0.8}{$
   \mbox{Reject }H_{0}: \mathbf{\Sigma}_{1}= \mathbf{\Sigma}_{2}  \mbox{ if } P(\mathbf{F} >\mathbf{F}_{c}) \equiv \mbox{ p-value}  < \alpha,
$}\\[1ex]
\end{center}
where
$$
  \mathbf{F}_{c} = \mathbf{S}_{2}^{-1/2} \mathbf{S}_{1} \mathbf{S}_{2}^{-1/2},
$$
the $p$-value follow from corollaries \ref{cor1} or \ref{cor2}, and $\mathbf{S}_{1}$ and $\mathbf{S}_{2}$ are the sample variances-covarianzas matrices of the two samples. 
\medskip
\begin{example}
Four psychological tests were given to 32 men and 32 women. The data are recorded in \citet[Table 5.1, p. 125]{re:02}. The variables are
$y_{1}$ = pictorial inconsistencies, $y_{2}$ = paper form board, $y_{3}$ = tool recognition, and  $y_{4}$ = vocabulary, see \citet[Example 5.4.2, p124]{re:02}. We are interesting in  test the hypothesis $H_{0}: \mathbf{\Sigma}_{1}= \mathbf{\Sigma}_{2}$ versus $H_{1}: \mathbf{\Sigma}_{1} \neq \mathbf{\Sigma}_{2}$. The sample variance-covariance matrices $\mathbf{S}_{1}$ and $\mathbf{S}_{2}$ are given in \citet[Example 5.4.2, p. 124]{re:02}, from where the matrix $\mathbf{F}_{c}$ and its eigenvalues are given by 
$$
  \mathbf{F}_{c} = \begin{pmatrix*}[r]
            0.5164511 & -0.1089194  & 0.2211275  & 0.1108078\\
           -0.1089194 &  0.7934331  & -0.1813041 & 0.0948122\\
            0.2211275 & -0.1813041  &  0.9451825 & 0.1474816\\
            0.1108078 &  0.0948122  &  0.1474816 & 0.4676369
       \end{pmatrix*},
$$
$$
   (\lambda_{1}, \lambda_{2},\lambda_{3}, \lambda_{4}) =(1.1773492, 0.7635739, 0.4493134, 0.3324671).
$$ 
The six expressions (\ref{pvBI1}) to (\ref{pvBII3}) computed the $p$-value with the following results: the probabilities (\ref{pvBI1}), (\ref{pvBII1}) and (\ref{pvBII3}) diverge by different reasons, and the other three upper probabilities lead to the following decision rule: 
\begin{center}
   {\bf do not reject the null hypothesis by $\mathbf{p}$-value = 0.0585654.}
\end{center}

This decision based on exact matrix probability is not in agreement with \citet[pp. 258-259]{re:02}, because, the Rencher tests are based on approximations of the distributions of the three test statistics used there. Therefore, two clarifications are considered: i) The decision made in Rencher shall depend on the behavior of the approximations of the distributions of the test statistics in this particular example. ii) On the other hand, our $p$-value = 0.0585654, is very close to being significant, if we consider $\alpha = 0.05$. A  classical paradigm in hypothesis testing.
\end{example}

\section*{Conclusions}

This work has provided the unified theory of real normed division algebras for the foundations of distributional and matrix probabilities that launch a natural and promising definition of matrix $p-$values for diverse hypothesis testing, such as MANOVA. Testing equality of covariance matrices is also a feasible future application.

%\section*{Acknowledgements}

%The authors wish to thank the Editor and the anonymous reviewers for their constructive comments
%on the preliminary version of this paper.

\end{document}